\documentclass[11pt]{amsart}

\usepackage{amsmath,amssymb,amsthm,mathtools}
\usepackage{enumitem}
\usepackage{microtype}
\usepackage[hidelinks]{hyperref}

\newtheorem{theorem}{Theorem}[section]
\newtheorem{proposition}[theorem]{Proposition}
\newtheorem{lemma}[theorem]{Lemma}
\newtheorem{corollary}[theorem]{Corollary}
\theoremstyle{remark}
\newtheorem{remark}[theorem]{Remark}

\newcommand{\R}{\mathbb R}
\newcommand{\Z}{\mathbb Z}
\newcommand{\SO}{\mathrm{SO}}
\newcommand{\OO}{\mathrm O}
\newcommand{\GH}{\mathrm{GH}}
\newcommand{\diam}{\operatorname{diam}}
\newcommand{\Vol}{\operatorname{Vol}}
\newcommand{\Ric}{\operatorname{Ric}}
\newcommand{\Hol}{\operatorname{Hol}}
\newcommand{\eps}{\varepsilon}

\title[Singular Limits of Frame Bundles]{Singular Limits of Frame Bundles of Noncollapsing Positively Curved Manifolds}

\author[Chao Qian]{Chao Qian}
\address{School of Mathematics and Statistics, Beijing Institute of Technology, Beijing 100081, P.R. China}
\email{6120150035@bit.edu.cn}
\author[Wilderich Tuschmann]{Wilderich Tuschmann}
\address{Karlsruhe Institute of Technology, Karlsruhe 76131, Germany}
\email{wilderich.tuschmann@kit.edu}
\thanks{The project is partially supported by the NSFC (No. 12371048).}
\makeatletter
\@ifundefined{subjclassname@2020}{%
  \@namedef{subjclassname@2020}{\textup{2020} Mathematics Subject Classification}}{}
\makeatother
\date{}

\begin{document}

\begin{abstract}
For each dimension $n\ge 2$, we construct noncollapsing sequences of
metrics on $S^n$ with a uniform positive lower bound on sectional curvature
whose oriented orthonormal frame bundles, equipped with the canonically
induced connection metrics, converge to nonmanifold metric spaces
homeomorphic to \(\Sigma(S^{n-1}\times\SO(n))\).
In even dimensions, real projective quotients give connected full frame
bundles with the same limits, disproving the manifold conclusion in
Yamaguchi's frame bundle conjecture. In odd dimensions, each component of
the full projective frame bundle, with its intrinsic metric, converges to
the quotient of the spherical limit by a free involution. The singularities are detected by local homology, and
we obtain quantitative equivariant Gromov--Hausdorff estimates. We also
compare the frame bundle limits with the intrinsic completions of their
regular limiting frame metrics, and present and discuss alternatives to Yamaguchi's original conjecture.
\end{abstract}

\maketitle

\section{Introduction}

Perelman's Stability Theorem (see \cite{Kap}) and Yamaguchi's Fibration Theorem \cite{Yam}
are cornerstones of the theory of convergence and collapse under a lower bound on sectional curvature.
Yamaguchi's work also led to the conjectural picture that Gromov--Hausdorff limits $Y$ of naturally
metrized orthonormal frame bundles of Gromov--Hausdorff-convergent
sequences $M_i \to X$ of Riemannian $n$-manifolds $M_i$ with a uniform lower bound on sectional
curvature and a uniform upper bound on diameter should be topological manifolds (see  \cite[p. 323]{Yam}).
For connected full frame bundles, equivariant convergence yields the
metric quotient representation $X\cong Y/\OO(n)$ independently of the
conjectural manifold conclusion. The additional assertion that $Y$ is a
manifold would constrain the possible singularities of $X$.
With an additional uniform upper bound on sectional curvature, one enters
the setting of Cheeger--Fukaya--Gromov theory \cite{CFG}. Its regularity
statements for frame-bundle constructions must be interpreted with the
appropriate smoothing or curvature-derivative hypotheses \cite{Fuk,PT}.
In particular, these statements do not directly establish
$C^{1,\alpha}$ regularity for limits of the unsmoothed canonical connection
metrics used here (see also \cite[Remark~1.5(2)]{SX}).

In the present note, we show that the manifold conclusion for $Y$ fails for the natural connection metrics on the frame bundles,
even for noncollapsing sequences of Riemannian manifolds with a
uniform positive lower bound on sectional curvature.


Let $M = (M^n,g)$ be a Riemannian manifold of dimension $n$ and
let $FM=F(M^n,g)$ denote its $\OO(n)$ principal bundle of orthonormal frames.
If \(M\)
is oriented, let \(F^+(M,g)\subset FM\) denote the bundle of
positively oriented orthonormal frames, a principal \(\SO(n)\) bundle.


After fixing
a bi-invariant metric on $\OO(n)$, the Levi--Civita connection of $g$ gives rise
to a canonical metric on $FM$, for which the projection $FM \to M$ is
a Riemannian submersion and $O(n)$ acts on $FM$ by isometries.
Throughout the paper, we equip the frame bundle $\OO(n)\to F(M)\to M$ of
\((M,g)\) with the Riemannian metric
\begin{equation}\label{eq:connection-metric}
 h_g=\pi^*g+b(\omega_g,\omega_g),
 \qquad b(U,V)=-\frac12\operatorname{tr}(UV),
\end{equation}
canonically induced from $g$. Here \(\omega_g\) is the Levi--Civita connection form,
and $b$ is the fixed $\operatorname{Ad}(\OO(n))$-invariant inner product
on $\mathfrak{so}(n)$ inducing the bi-invariant fiber metric.
Thus horizontal lifts have the same length as their projections, while
each connected component of a frame fiber retains its fixed intrinsic
Riemannian metric. We always use the Levi--Civita connection of the
corresponding base metric, with a fixed vertical normalization and without
separately smoothing the horizontal distribution.
When the full frame bundle is disconnected, we use the intrinsic length
metric on each component separately; the Riemannian metric alone does
not specify a finite distance between different components.

To describe our results in detail, let \(X_n\) be the metric completion of
\[
 X_n^\circ=(0,\pi)\times S^{n-1}
\]
equipped with the ``spindle'' metric
\[
 g_0=dr^2+\frac14\sin^2r\,g_{S^{n-1}}.
\]
Thus \(X_n\) is obtained by adjoining one point at each end and is
homeomorphic to the suspension \(\Sigma S^{n-1}\cong S^n\).
Using rotationally symmetric metrics, we then obtain a noncollapsing sequence of positively curved metrics on $S^n$ which has a frame-bundle limit with genuine topological singularities, even though the base limit remains a topological sphere $S^n$. The metrics on $S^n$ converge to the spindle $X_n$ with two conical metric singularities at its poles. The smooth approximating metrics round off these singularities in shrinking neighborhoods of the poles. More concretely, define the polar caps by the geodesic balls of radius $\eps$ about the north and south poles, given in radial coordinates by $r\le\eps$ and $\pi-r\le\eps$ respectively. The curvature becomes unbounded in these caps as $\eps\to 0$. The key mechanism we employ is short-loop holonomy: at either pole, every element of \(\mathrm{SO}(n)\) is realized by parallel transport around a piecewise smooth loop based at that pole, contained in the corresponding cap, and of length at most \(C_n\varepsilon\). The horizontal lift of such a loop joins two frames in the same fiber and has the same length as the base loop. Consequently, although a frame fiber does not shrink intrinsically, its ambient diameter in the frame bundle tends to
zero.

This technique is consistent both with Sol\'orzano's ``wane group''
framework for Sasaki-type vector bundles \cite{Sol} and with the
infinitesimal-holonomy phenomena studied by Si--Xu \cite{SX}.
The new
features here are compactness, positive lower bounds on sectional curvature,
noncollapse, explicit global limits in every dimension, local homology
obstructions, and quantitative convergence.

Set
\[
 L_n=S^{n-1}\times\SO(n),
 \qquad
 m=n+\dim\SO(n)=\frac{n(n+1)}2.
\]
Thus \(m\) is the dimension of an oriented orthonormal frame bundle
over an \(n\)-manifold (and of the regular part of the limiting frame space).
The precise statement of
our main result is the following.

\begin{theorem}[Spherical examples]\label{thm:sphere}
For every \(n\ge2\), there are smooth antipodally invariant metrics
\(\widehat g_\eps\) on \(S^n\), \(0<\eps<1/4\), such that
\[
 \sec_{\widehat g_\eps}\ge\frac1\pi,
 \qquad
 \diam(S^n,\widehat g_\eps)=\pi,
 \qquad
 \Vol(S^n,\widehat g_\eps)
 \ge 2^{-(n-1)}\Vol(S^n,g_{\mathrm{round}}).
\]

As $\eps\to 0$, these manifolds converge to \(X_n\) in the
Gromov--Hausdorff topology, quantitatively:
\[
 d_{\GH}\bigl((S^n,\widehat g_\eps),X_n\bigr)
 \le C\sqrt\eps.
\]
This convergence may be realized by Gromov--Hausdorff approximations which are equivariant with
respect to the antipodal involutions.

Equip the positively oriented frame bundles
\[
 P_\eps=F^+(S^n,\widehat g_\eps)
\]
with the Levi--Civita connection metric
\[
 h_\eps
 =\pi_\eps^*\widehat g_\eps
  +b(\omega_\eps,\omega_\eps),
 \qquad
 b(U,V)=-\frac12\operatorname{tr}(UV).
\]
Then there exist a compact length space \(Y_n\) and a constant \(C_n\)
such that
\[
 d_{\GH}\bigl((P_\eps,h_\eps),Y_n\bigr)
 \le C_n\sqrt\eps.
\]
Topologically,
\[
 Y_n\cong\Sigma L_n
 =\Sigma\bigl(S^{n-1}\times\SO(n)\bigr).
\]
The convergence may be realized by Gromov--Hausdorff approximation maps
from \(P_\eps\) to \(Y_n\) that are equivariant under the right
\(\SO(n)\) action. For odd \(n\), these maps may also be chosen
equivariant under the natural lift of the antipodal involution.

The regular part of \(Y_n\) is a smooth \(m\)-manifold.  Its two
suspension vertices \(v_N,v_S\) are its only nonmanifold points, and
at either vertex \(v\),

\begin{equation}
\begin{aligned}\label{eq:main-homology}
 H_n(Y_n,Y_n\setminus\{v\};\Z)
   &\cong
   \Z\oplus H_{n-1}(\SO(n);\Z),\\
 H_m(Y_n,Y_n\setminus\{v\};\Z)
   &\cong\Z.
 \end{aligned}
\end{equation}

The right \(\SO(n)\) action extends to \(Y_n\) by fixing
\(v_N\) and \(v_S\).  The projection
\[
 (r,u,Q)\longmapsto(r,u)
\]
on the regular part, with \(v_N,v_S\) sent to the corresponding
poles of \(X_n\), induces an isometry
\[
 Y_n/\SO(n)\xrightarrow{\;\cong\;}X_n,
\]
where \(Y_n/\SO(n)\) is equipped with its metric quotient distance.
\end{theorem}
\begin{remark}
For reference, the class motivating Yamaguchi's original conjecture is
\begin{equation}\label{eq:yamaguchi-class}
 \sec(M_i)\ge-1,
 \qquad \diam(M_i)\le D.
\end{equation}
After a fixed rescaling, however, our examples lie in a considerably smaller class: they can actually be chosen with \(\sec\ge1\), any prescribed upper bound on diameter, and a uniform positive lower bound on volume;
compare Corollary~\ref{cor:normalization}. Moreover, the curvature operators of the metrics \(\widehat g_\varepsilon\) constructed in Theorem~\(\ref{thm:sphere}\) actually also admit a uniform positive lower bound on their eigenvalues.
\end{remark}

Antipodal invariance of the metrics \(\widehat g_\eps\) also gives
the following.

\begin{theorem}[Projective space examples]\label{thm:projective}
Let \(g_\eps\) be the metric induced from \(\widehat g_\eps\)
on \(\mathbb{RP}^n\).
If \(n\) is even, the connected full frame bundle
\(F(\mathbb{RP}^n,g_\eps)\) is isometric to
\(F^+(S^n,\widehat g_\eps)\), and it converges to \(Y_n\)
in the Gromov--Hausdorff topology as $\eps\to 0$.
If \(n\) is odd, the full frame bundle has two components equipped
separately with their intrinsic component metrics, and
for $\eps\to 0$ each component converges to
\[
 Z_n=Y_n/\langle\tau\rangle,
\]
where \(\tau\) is a free isometric involution exchanging the two
vertices. The space \(Z_n\) has exactly one nonmanifold point, and
the manifolds \((\mathbb{RP}^n,g_\eps)\) converge in the Gromov--Hausdorff topology to \(Z_n/\SO(n)\).
\end{theorem}

\section{The rotation-invariant base metrics}

Choose a smooth nonincreasing function
\(\chi:[0,\infty)\to[0,1]\) satisfying
\[
 \chi\equiv1\quad\text{on }[0,1/2],
 \qquad
 \chi\equiv0\quad\text{on }[1,\infty).
\]
Put \(a=1/2\). For \(0\le r\le\pi/2\), define
\begin{equation}\label{eq:profile}
 f_\eps(r)
 =a\sin r+(1-a)\eps\int_0^{r/\eps}\chi(t)\,dt,
\end{equation}
and extend \(f_\eps\) by reflection by setting \(f_\eps(\pi-r)=f_\eps(r)\). Set
\begin{equation}\label{eq:base-metric}
 \widehat g_\eps=dr^2+f_\eps(r)^2g_{S^{n-1}}.
\end{equation}
The limiting profile on the open cylinder is
\[
 f_0(r)=a\sin r=\frac12\sin r.
\]
On \(\eps\le r\le\pi/2\),
\begin{equation}\label{eq:constant}
 c_\eps=(1-a)\eps\int_0^1\chi(t)\,dt,
 \qquad 0<c_\eps\le\eps/2,
\end{equation}
and
\begin{equation}\label{eq:middle-profile}
 f_\eps=f_0+c_\eps,
 \qquad f_\eps'=f_0'.
\end{equation}

\begin{proposition}\label{prop:base-geometry}
The metric \(\widehat g_\eps\) extends smoothly over both poles and
across the equator. It is antipodally invariant and satisfies the curvature,
diameter, and volume estimates in Theorem~\ref{thm:sphere}.
\end{proposition}

\begin{proof}
Near \(r=0\), where \(\chi=1\),
\[
 f_\eps(r)=\frac12(\sin r+r)
 =r-\frac{r^3}{12}+O(r^5).
\]
Thus \(f_\eps\) has a smooth odd extension with
\(f_\eps(0)=0\) and \(f_\eps'(0)=1\). Equivalently,
\((f_\eps(r)/r)^2\) is smooth and even, so the standard Cartesian
formula proves smooth extension at the pole. Reflection handles the
southern pole. Near the equator, \(f_\eps=\frac12\sin r+c_\eps\), which
is even about \(\pi/2\), and hence the reflected profile is smooth.

On the northern hemisphere,
\[
 f_\eps'=\frac12\cos r+\frac12\chi(r/\eps),
 \qquad
 f_\eps''=-\frac12\sin r+\frac1{2\eps}\chi'(r/\eps).
\]
In particular, \(0\le f_\eps'\le1\), \(f_\eps\le r\), and
\[
 f_\eps''+\frac1\pi f_\eps
 \le-\frac12\sin r+\frac r\pi\le0
 \qquad(0\le r\le\pi/2).
\]
The warped-product curvature formulas are
\[
 K_{\rm rad}=-\frac{f_\eps''}{f_\eps},
 \qquad
 K_{\rm tan}=\frac{1-(f_\eps')^2}{f_\eps^2}.
\]
Thus \(K_{\rm rad}\ge1/\pi\). Moreover,
\[
 \frac d{dr}\left((f_\eps')^2+\pi^{-1}f_\eps^2\right)
 =2f_\eps'(f_\eps''+\pi^{-1}f_\eps)\le0.
\]
Since the expression equals \(1\) at \(r=0\),
\[
 (f_\eps')^2+\pi^{-1}f_\eps^2\le1,
\]
and therefore \(K_{\rm tan}\ge1/\pi\). In dimension two, only the
radial sectional curvature occurs. The inequalities extend across the smooth
poles by continuity.

The radial coordinate is $1$-Lipschitz and a meridian from the north pole to the south pole
has length \(\pi\), so the pole distance is \(\pi\). Any two points can
be joined through one of the poles by a path of length
at most \(\pi\); hence the
diameter is exactly \(\pi\). Finally,
\(f_\eps\ge\frac12\sin r\), and integration of the volume form gives
\[
 \Vol(S^n,\widehat g_\eps)
 \ge2^{-(n-1)}\Vol(S^n,g_{\rm round}).
\]
The map
\[
 A(r,u)=(\pi-r,-u)
\]
is an isometry. It is the ordinary antipodal map and has degree
\((-1)^{n+1}\).
\end{proof}

Notice that the curvature is not uniformly bounded from above. At any transition point
\(t\eps\) with \(\chi'(t)<0\),
\[
 K_{\rm rad}(t\eps)=c(t)\eps^{-2}+O(1),
 \qquad c(t)>0.
\]
This curvature concentration is the mechanism that produces order-one
holonomy around loops of length \(O(\eps)\).

\section{The frame bundle metric and short-loop holonomy}

On the punctured sphere, identify
\(T((0,\pi)\times S^{n-1})\) with
\(\R\oplus TS^{n-1}\). Define a global orthonormal trivialization
\begin{equation}\label{eq:trivialization}
 \Phi_f(r,u)v
 =\langle v,u\rangle\partial_r
 +f(r)^{-1}\bigl(v-\langle v,u\rangle u\bigr),
 \qquad v\in\R^n.
\end{equation}
It trivializes \(\R\oplus TS^{n-1}\) and does not require
\(TS^{n-1}\) itself to be trivial. We choose the orientation of \(S^n\)
for which \(\Phi_f\) is a positive frame on the punctured sphere.
In this trivialization,
the connection matrix
is
\begin{equation}\label{eq:connection-form}
 A_f=(f'-1)(du\,u^T-u\,du^T).
\end{equation}
The matrix \(du\,u^T-u\,du^T\) is skew-symmetric. Writing a positively
oriented frame as \(\Phi_fQ\), \(Q\in\SO(n)\), gives
\begin{equation}\label{eq:frame-metric}
 \begin{aligned}
 \omega_f&=Q^{-1}dQ+Q^{-1}A_fQ,\\
 h_f&=dr^2+f^2g_{S^{n-1}}
 +b\bigl(Q^{-1}dQ+Q^{-1}A_fQ,
         Q^{-1}dQ+Q^{-1}A_fQ\bigr).
 \end{aligned}
\end{equation}
By \eqref{eq:middle-profile}, the connection forms for \(f_\eps\) and
\(f_0\) agree exactly on the middle annulus. Only the horizontal
angular coefficient changes there.

\begin{lemma}[Short-loop holonomy]\label{lem:cap}
Every element of \(\SO(n)\) is
realized as the holonomy of a loop based at either pole and having length at most
\[
 K_n\eps,
 \qquad K_n=(2+2\pi)\lfloor n/2\rfloor.
\]
Consequently,
\[
 \diam_{P_\eps}\pi_\eps^{-1}(N)\le K_n\eps
\]
and, for every \(\delta\ge\eps\),
\[
 \diam_{P_\eps}\pi_\eps^{-1}\{r\le\delta\}
 \le2\delta+K_n\eps.
\]
The analogous estimate holds at the southern pole.
\end{lemma}

\begin{proof}
Fix a two-dimensional subspace \(E\subset T_NS^n\).
The rotational $2$-sphere
corresponding to \(E\) is totally geodesic: it is the relevant fixed
component of the fixed-point set of the isometry acting as \(+I\) on \(E\) and \(-I\) on
\(E^\perp\). Normal vectors in \(E^\perp\) extend as parallel fields,
so the holonomy of every loop in this \(2\)-sphere
acts trivially on \(E^\perp\).

Consider the loop that runs radially from the north pole to \(r=s\),
once around that latitude, and radially back. Define the real-valued
angle
\[
 \widetilde\alpha_\eps(s)=2\pi(1-f_\eps'(s))\in\R,
 \qquad
 \alpha_\eps(s)=\widetilde\alpha_\eps(s)\pmod{2\pi}.
\]
Parallel transport acts on \(E\) by rotation through
\(\alpha_\eps(s)\). Since
\[
 \widetilde\alpha_\eps(0)=0,
 \qquad
 \widetilde\alpha_\eps(\eps)
 =2\pi\left(1-\frac12\cos\eps\right)\ge\pi,
\]
continuity shows that every angle in \([0,\pi]\) is realized,
and reversing the loop realizes every angle in \([-\pi,0]\). The loop length is
\[
 2s+2\pi f_\eps(s)\le(2+2\pi)\eps.
\]
Every element of \(\SO(n)\) is a product of at most
\(\lfloor n/2\rfloor\) planar rotations. Concatenating the
corresponding loops proves the first claim.

Given two frames over \(\{r\le\delta\}\), lift radial curves from their
base points to the pole, use the horizontal lift of a short base loop there to
match the resulting frames, and
traverse the other radial lift in reverse. The
length is at most \(2\delta+K_n\eps\). Reflection yields the
corresponding statement for the southern part.
\end{proof}

\section{Metric contraction and the base limit}

We use the following elementary quotient estimate.

\begin{lemma}\label{lem:contract}
Let \(Z\) be a compact length space and let
\(A_1,\dots,A_k\subset Z\) be pairwise disjoint compact sets. Collapse
each \(A_j\) to a point and equip the quotient \(\bar Z\) with the
induced chain metric. Then \(\bar Z\) is a compact length space, its
metric topology is the ordinary quotient topology, and
\[
 0\le d_Z(x,y)-d_{\bar Z}(\bar x,\bar y)
 \le\sum_{j=1}^k\diam_ZA_j.
\]
In particular,
\[
 d_{\GH}(Z,\bar Z)
 \le\frac12\sum_{j=1}^k\diam_ZA_j.
\]
\end{lemma}

\begin{proof}
An admissible chain alternates ordinary paths in \(Z\) with zero-cost
jumps whose endpoints lie in the same \(A_j\); its cost is the sum of
the ordinary path lengths. Fix \(\eta>0\) and choose a chain from
\(x\) to \(y\) of cost at most
\(d_{\bar Z}(\bar x,\bar y)+\eta\), having the fewest possible
zero-cost jumps among all such chains. If the chain uses two zero-cost jumps in the same \(A_j\), replace the
entire subchain from the start of the first such jump to the end of the
last such jump by one
zero-cost jump in \(A_j\). This does not increase the cost and decreases
the number of jumps, a contradiction. Hence every \(A_j\) is used at
most once. Replacing its jump by a path in \(Z\) of length at most
\(\diam A_j+\eta/k\), and then letting \(\eta\downarrow0\), proves the
distance estimate. The reverse inequality follows because the
quotient map is 1-Lipschitz.

To prove separation, fix distinct quotient classes \(\bar x,\bar y\).
For each choice of a subset of \(\{A_1,\dots,A_k\}\) and each ordering
of that subset, consider chains that visit exactly those contracted
sets in the prescribed order. Their cost is a continuous function on
a finite product of compact sets, and hence attains its minimum. If
this minimum were zero, every ordinary-distance term would vanish.
Consecutive quotient classes would then intersect, forcing all classes
in the chain, including \(\bar x\) and \(\bar y\), to coincide. Since
there are finitely many patterns, the minimum over them is positive.
Thus the chain pseudometric is a genuine metric.

Let \(Z/{\sim}\) be the ordinary topological quotient. The 1-Lipschitz
quotient map induces a continuous bijection
\[
 Z/{\sim}\longrightarrow(\bar Z,d_{\bar Z}).
\]
The source is compact and the target is Hausdorff, so it is a
homeomorphism. Compactness follows, and almost-minimizing chains show
that the quotient is a length space. Finally, the graph of the
quotient map is a correspondence whose distortion is bounded by the
sum of the diameters, proving the Gromov--Hausdorff estimate.
\end{proof}

Let \(X_n\) be the metric completion of
\[
 (0,\pi)\times S^{n-1},
 \qquad
 g_0=dr^2+\frac14\sin^2r\,g_{S^{n-1}}.
\]
It adds one point at each end and is topologically
\(\Sigma S^{n-1}\cong S^n\).

\begin{proposition}\label{prop:base-limit}
There is a universal constant \(C\) such that
\[
 d_{\GH}\bigl((S^n,\widehat g_\eps),X_n\bigr)
 \le C\sqrt\eps.
\]
The convergence is antipodally equivariant.
\end{proposition}

\begin{proof}
Fix \(\eps\le\delta\le1/2\) and contract the closed northern and
southern \(\delta\)-caps. Their diameters are at most \(2\delta\). On
the remaining annulus, \eqref{eq:middle-profile} gives
\[
 1\le\frac{f_\eps}{f_0}
 \le1+C\frac\eps\delta,
\]
because \(\sin r\ge c\delta\) for
\(r\in[\delta,\pi-\delta]\), with a universal constant \(c>0\).
Thus the annular metrics are bi-Lipschitz with Lipschitz constant
\(1+C\eps/\delta\). After contraction, any excursion into a polar cap
can be replaced by a zero-cost jump. The quotient metric is therefore
the annular length metric with each boundary component collapsed to a
point. On the common quotient, the two distances satisfy
\[
 d_0\le d_\eps\le(1+C\eps/\delta)d_0.
\]
The two radial shortcuts give diameter at most \(\pi-2\delta\) for
the limiting quotient, so the identity correspondence has distortion
at most \(C(\eps/\delta)(\pi-2\delta)\).
Applying Lemma~\ref{lem:contract} on both sides yields
\[
 d_{\GH}\bigl((S^n,\widehat g_\eps),X_n\bigr)
 \le4\delta+C\frac\eps\delta.
\]
Setting \(\delta=\sqrt\eps\) gives the result. All contractions and annular identifications commute with the antipodal map.
\end{proof}

\section{The frame bundle limit}

Using the trivialization \eqref{eq:trivialization}, identify the regular frame space with \(P^\circ=(0,\pi)\times L_n\).
Let \(h_0\) be
the metric \eqref{eq:frame-metric} associated to \(f_0\). Denote its
intrinsic distance by \(d^\circ\). Add two points \(v_N,v_S\), and for
\(x,y\in P^\circ\) set
\begin{equation}\label{eq:limit-distance}
 d_Y(x,y)=\min\left\{
 d^\circ(x,y),\ r_x+r_y,\ 2\pi-r_x-r_y
 \right\}.
\end{equation}
Also set
\[
 d_Y(v_N,x)=r_x,
 \qquad d_Y(v_S,x)=\pi-r_x,
 \qquad d_Y(v_N,v_S)=\pi.
\]

\begin{proposition}\label{prop:limit-topology}
These formulas define on \(Y_n=P^\circ\cup\{v_N,v_S\}\) a metric
that turns $Y_n$ into a compact length metric space.
Its topology is given by
\[
 Y_n\cong\Sigma L_n,
\]
and on the regular part the metric is locally the Riemannian metric \(h_0\).
\end{proposition}

\begin{proof}
The displayed formula is the path-chain metric obtained by allowing ordinary
\(h_0\)-paths and zero-cost changes in all angular and frame coordinates at
either added vertex. A chain that visits one vertex has a cost bounded
below by the corresponding radial shortcut. If it visits both, then
in the order north--south its cost is at least
\[
 r_x+\pi+(\pi-r_y)\ge r_x+r_y,
\]
and in the opposite order it is at least
\[
 (\pi-r_x)+\pi+r_y\ge2\pi-r_x-r_y.
\]
Thus no more complicated chain improves the displayed minimum. This
also proves the triangle inequality.

For \(0<\delta<\pi/2\),
\[
 B(v_N,\delta)=\{v_N\}\cup((0,\delta)\times L_n),
\]
and similarly at \(v_S\). Hence the topology is the suspension
topology. Compactness follows from compactness of \([0,\pi]\times L_n\)
and the quotient description. The chain construction provides $Y_n$ with the structure of a length
space, and on
every relatively compact normal neighborhood in the regular part
the radial shortcut costs are bounded away from zero, so the metric agrees
locally with the Riemannian metric \(h_0\).
\end{proof}

\begin{theorem}\label{thm:frame-convergence}
For \(0<\eps<1/4\),
\[
 d_{\GH}(P_\eps,Y_n)\le C_n\sqrt\eps.
\]
The Gromov--Hausdorff approximation maps from \(P_\eps\) to \(Y_n\)
may be chosen equivariant with respect to the right \(\SO(n)\) action.
For odd \(n\), they may also be chosen equivariant under the natural
lift of the antipodal involution.
\end{theorem}

\begin{proof}
Assume first \(\eps<1/16\) and put \(\delta=\sqrt\eps\). Contract the
two closed polar frame caps in \(P_\eps\) and denote the quotient space
by $Q_{\eps,\delta}$. Lemmas~\ref{lem:cap} and
\ref{lem:contract} give
\[
 d_{\GH}(P_\eps,Q_{\eps,\delta})
 \le2\delta+K_n\eps.
\]
On the middle annulus the connection forms agree exactly. Only the
horizontal angular coefficient changes, and
\[
 h_0\le h_\eps
 \le\left(1+C\frac\eps\delta\right)^2h_0.
\]
Let \(Q_{0,\delta}\) be the corresponding annular quotient equipped
with the length metric induced by \(h_0\), with each boundary component
collapsed separately. As in the proof of Proposition~\ref{prop:base-limit},
cap excursions can be replaced by zero-cost jumps. Under the coordinate
identification, the quotient distances satisfy
\[
 d_{Q_{0,\delta}}\le d_{Q_{\eps,\delta}}
 \le(1+C\eps/\delta)d_{Q_{0,\delta}}.
\]
The radial shortcuts imply
\(\diam Q_{0,\delta}\le\pi-2\delta\). Consequently the identity
correspondence has distortion at most
\(C(\eps/\delta)\diam Q_{0,\delta}\), and hence
\[
 d_{\GH}(Q_{\eps,\delta},Q_{0,\delta})
 \le C_n\frac\eps\delta.
\]
The two \(\delta\)-caps in \(Y_n\) have diameter at most \(2\delta\),
so contracting them changes the Gromov--Hausdorff distance by at most \(2\delta\). Hence
\[
 d_{\GH}(P_\eps,Y_n)
 \le4\delta+K_n\eps+C_n\frac\eps\delta
 \le C_n\sqrt\eps.
\]

If \(1/16\le\eps<1/4\), the full product
\(P_\eps\times Y_n\) is a correspondence of distortion at most
\(\max\{\diam P_\eps,\diam Y_n\}\). Horizontally lifting a minimizing
base path and then moving in its terminal fiber gives
\[
 \diam P_\eps\le\pi+\diam(\SO(n),b),
 \qquad \diam Y_n=\pi.
\]
Since \(\sqrt\eps\ge1/4\), increasing \(C_n\) proves the estimate
on the remaining interval.

For explicit approximation maps, send the northern and southern closed
\(\delta\)-caps of \(P_\eps\) to \(v_N\) and \(v_S\), respectively,
and use the coordinate identification on the remaining open annulus.
For \(\eps<1/16\), take \(\delta=\sqrt\eps\). The contraction and
quotient comparisons above bound the distortion by \(C_n\sqrt\eps\),
and the image is \(\delta\)-dense. For the remaining values of \(\eps\),
take \(\delta=1/4\) and use the uniform diameter bounds.
These maps commute with the right \(\SO(n)\) action. For odd \(n\),
the antipodal map preserves orientation and its natural lift exchanges
the caps and commutes with the coordinate identification, so the maps
are also antipodally equivariant. The approximation maps are taken in
the direction \(P_\eps\to Y_n\): no exact \(\SO(n)\)-equivariant map
\(Y_n\to P_\eps\) can exist, since the vertices of \(Y_n\) are fixed
points and the action on \(P_\eps\) is free.
\end{proof}

\begin{proposition}\label{prop:spherical-quotient}
The coordinate projection induces an isometry
\[
 Y_n/\SO(n)\cong X_n.
\]
\end{proposition}

\begin{proof}
On the regular part, projection does not increase \(d^\circ\). In the
opposite direction, every piecewise smooth base path has a horizontal
lift from any prescribed initial frame, and the lift has the same
length.
Passing to the quotient over the terminal frame fiber removes any dependence on the endpoint frame.
Thus the quotient of \(d^\circ\) is precisely
the intrinsic distance on the punctured spindle. The two radial
shortcut terms in \eqref{eq:limit-distance} are unchanged by the group
action. Consequently the quotient distance is
\[
 \min\left\{
 d_{X_n^\circ}((r,u),(s,v)),\ r+s,\ 2\pi-r-s
 \right\},
\]
which is exactly the completed spindle distance.
The two suspension vertices map to the corresponding poles of the spindle.
\end{proof}

\section{Local homology and the proof of Theorem \ref{thm:sphere}}

For a compact space \(L\), write
\[
 C^\circ(L)=([0,1)\times L)/(\{0\}\times L)
\]
for the open cone with its vertex included.
For the basic algebraic-topological notions and results used
in this section, see \cite{Hat}.

\begin{proposition}\label{prop:local-homology}
The regular part of \(Y_n\) is an \(m\)-manifold. Each vertex has a
neighborhood homeomorphic to \(C^\circ(L_n)\), and its local homology is
as in \eqref{eq:main-homology}. In particular, neither vertex is a
manifold point, either in the interior or on the boundary.
\end{proposition}

\begin{proof}
The regular dimension is
\[
 1+(n-1)+\frac{n(n-1)}2=\frac{n(n+1)}2=m.
\]
The suspension description gives
\[
 H_k(Y_n,Y_n\setminus\{v\};\Z)
 \cong\widetilde H_{k-1}(L_n;\Z).
\]
Because the homology of \(S^{n-1}\) is free, the K\"unneth formula yields
\[
 H_{n-1}(L_n;\Z)
 \cong H_{n-1}(\SO(n);\Z)\oplus\Z.
\]
The connected Lie group \(\SO(n)\) is orientable, so the connected
closed \((m-1)\)-manifold \(L_n\) is orientable and
\[
 H_{m-1}(L_n;\Z)\cong\Z.
\]
This proves \eqref{eq:main-homology}. Since \(n<m\), local homology is
nonzero in two distinct degrees, whereas an interior point of an
\(m\)-manifold has local homology only in degree \(m\). A boundary point
has vanishing local homology, as in
\[
 H_*(\R^m_{\ge0},\R^m_{\ge0}\setminus\{0\};\Z)=0.
\]
Thus the vertices are not manifold points.
\end{proof}

For \(n=2\), \(L_2\cong T^2\), \(m=3\), and the two relevant local
groups are \(\Z^2\) in degree two and \(\Z\) in degree three.

\begin{proof}[Proof of Theorem~\ref{thm:sphere}]
Proposition~\ref{prop:base-geometry} proves the geometric bounds and
Proposition~\ref{prop:base-limit} proves the base convergence.
Moreover, the convergence of frame bundles, the suspension topology, and the metric quotient are obtained in
Theorem~\ref{thm:frame-convergence} and
Propositions~\ref{prop:limit-topology} and~\ref{prop:spherical-quotient}, respectively.
Finally, Proposition~\ref{prop:local-homology} implies the remaining topological
assertions.
\end{proof}

\section{The projective space quotients}

Let \(A(r,u)=(\pi-r,-u)\). In the trivialization
\eqref{eq:trivialization},
\begin{equation}\label{eq:antipodal-frame}
 dA\circ\Phi_{\eps,r,u}
 =\Phi_{\eps,\pi-r,-u}\circ J_u,
 \qquad J_u=2uu^T-I.
\end{equation}
Here \(J_u\) is \(+1\) on \(\R u\) and \(-1\) on \(u^\perp\), so
\[
 \det J_u=(-1)^{n-1}.
\]
For even \(n\), the natural lift \(dA\) interchanges positive and
negative frames. If an involution of the oriented frame bundle is
desired, fix an orthogonal reflection \(R\) with \(R^2=I\) and define
\[
 T_\eps(p)=dA_{\pi_\eps(p)}\circ p\circ R,
 \qquad
 T(r,u,Q)=(\pi-r,-u,J_uQR).
\]
Both are isometric involutions; the latter extends to \(Y_n\) by
exchanging the vertices. They normalize the right \(\SO(n)\) action:
\[
 T_\eps(pB)=T_\eps(p)(RBR),\qquad B\in\SO(n),
\]
and the analogous identity holds for \(T\). In particular, they need
not commute with that action. The cap-collapse approximation maps
constructed in Theorem~\ref{thm:frame-convergence} intertwine
\(T_\eps\) and \(T\).

\begin{proof}[Proof of Theorem~\ref{thm:projective}]
If \(n\) is even, the antipodal map reverses orientation. Every frame
on \(\mathbb{RP}^n\) has two lifts to \(S^n\), exactly one of which is
positive. The covering differential therefore gives an isometric
diffeomorphism
\[
 F^+(S^n,\widehat g_\eps)
 \longrightarrow F(\mathbb{RP}^n,g_\eps).
\]
Thus the connected full frame bundle has limit \(Y_n\). Under this
identification, the right \(\OO(n)\) action is given on the regular part by
\[
 (r,u,Q)\cdot B=
 \begin{cases}
 (r,u,QB),&\det B=1,\\
 (\pi-r,-u,J_uQB),&\det B=-1.
 \end{cases}
\]
This formula also defines the isometric right action on \(Y_n\).
Determinant-negative elements exchange the two vertices, while the
vertex stabilizer is \(\SO(n)\). The approximation maps commute with
this action. Since the nonidentity coset in \(\OO(n)/\SO(n)\) induces
the antipodal map on \(X_n\), Proposition~\ref{prop:spherical-quotient}
also gives
\[
 Y_n/\OO(n)\cong X_n/\langle A\rangle.
\]
In particular, the full limiting fiber over the identified spindle pole
contains two distinct vertices, at distance \(\pi\); it is not collapsed
to one point.

If \(n\) is odd, the antipodal map preserves orientation and
\[
 F^+(\mathbb{RP}^n,g_\eps)
 \cong F^+(S^n,\widehat g_\eps)/\langle dA\rangle.
\]
The induced limit involution is
\[
 \tau(r,u,Q)=(\pi-r,-u,J_uQ).
\]
It is free: on the regular part a fixed point would require
\(u=-u\), and the two vertices are exchanged. Equivariant convergence
and passage to the quotient give convergence to
\(Z_n=Y_n/\langle\tau\rangle\). Two disjoint small cone neighborhoods
of the vertices are exchanged, so the quotient of either one maps
homeomorphically to a neighborhood of the single identified point.
Its local link is still \(L_n\), and Proposition~\ref{prop:local-homology}
shows that it is not a manifold point. No other point is singular.
The two components of the full frame bundle are isometric by right
multiplication with a reflection, proving the componentwise statement.
Finally, the right \(\SO(n)\) action commutes with \(\tau\), and
Proposition~\ref{prop:spherical-quotient} gives the claimed projective space
base quotient.
\end{proof}

\section{Normalization}

\begin{corollary}[Prescribed diameter examples]\label{cor:normalization}
For every \(D>0\) and every \(n\ge2\), the metrics $\widehat g_\eps$ and $g_\eps$
in Theorems~\ref{thm:sphere} and~\ref{thm:projective} may also be rescaled so
that
\[
 \sec\ge1,
 \qquad \diam\le D,
 \qquad \inf_\eps\Vol > v(n,D)>0,
\]
while all topological conclusions about the frame bundle limits remain valid.
Moreover, after this normalization,
\[
 \Ric_{g_\eps}\ge(n-1)g_\eps.
\]
In particular, in
every even dimension $n\ge 2$ the connected full orthonormal frame bundles of a noncollapsing
sequence of Riemannian $n$-manifolds $M_i$ satisfying $\sec (M_i) \ge -1$ and $\diam (M_i) \le D$
can have a nonmanifold limit.
\end{corollary}

\begin{proof}
Choose
\[
 \lambda=\min\left\{\frac D\pi,\frac1{\sqrt\pi}\right\}>0
\]
and replace every base metric by \(\lambda^2g_\eps\). Then
\[
 \sec_{\lambda^2g_\eps}
 =\lambda^{-2}\sec_{g_\eps}\ge1,
 \qquad
 \diam(\lambda^2g_\eps)=\lambda\diam(g_\eps)\le D.
\]
The volume lower bound is multiplied by \(\lambda^n\), and passage
from the sphere to its projective quotient divides volume by two. Thus
a common bound for both families after rescaling is
\[
 \inf_\eps\Vol\ge
 2^{-n}\lambda^n\Vol(S^n,g_{\rm round}),
\]
so we may take
\[
 v(n,D)=2^{-(n+1)}\lambda^n\Vol(S^n,g_{\rm round})>0.
\]
Constant rescaling does not change the Levi--Civita
connection. It multiplies horizontal lengths by \(\lambda\) and leaves
the fixed vertical metric unchanged, so the same contraction and
annular comparison arguments apply and the limit topology is
unchanged. The Ricci estimate follows from \(\sec\ge1\).
\end{proof}

\section{The intrinsic completion of the regular limit}

The ambient GH limit need not equal the intrinsic completion of the
regular limiting frame metric. The answer depends on the dimension.

\subsection{Dimension two}

Write \(u=e^{i\theta}\) and introduce the adjusted fiber angle
\(\varphi\) by
\[
 Q=e^{i(\varphi+\theta)}.
\]
Then
\begin{equation}\label{eq:surface-limit-metric}
 h_0
 =dr^2+\frac14\sin^2r\,d\theta^2
 +\left(d\varphi+\frac12\cos r\,d\theta\right)^2.
\end{equation}

\begin{proposition}\label{prop:surface-completion}
For \(n=2\), the intrinsic metric completion of
\((P^\circ,h_0)\) is isometric to the round
lens space
\[
 M=S^3(2)/\Gamma,
 \qquad
 \Gamma=\langle g\rangle\cong\mathbb Z_4,
 \qquad
 g(z_1,z_2)=(iz_1,-iz_2),
\]
where
\[
 S^3(2)=
 \{(z_1,z_2)\in\mathbb C^2:|z_1|^2+|z_2|^2=4\}.
\]
The completion adds two disjoint embedded closed geodesics
\[
 C_+=\{z_2=0\}/\Gamma,
 \qquad
 C_-=\{z_1=0\}/\Gamma,
\]
corresponding respectively to $r\downarrow0$
and $r\uparrow\pi$. They are parametrized by
\[
 \Psi_+=2\varphi+\theta\pmod{2\pi},
 \qquad
 \Psi_-=2\varphi-\theta\pmod{2\pi},
\]
respectively, and each has induced metric
$\frac14\,d\Psi_\pm^2$ and length $\pi$.
\end{proposition}

\begin{proof}
For $(r,\theta,\varphi) \in (0,\pi)\times\mathbb R^2$, define
\[
 \alpha_+=\frac{\varphi}{2}+\frac{\theta}{4},
 \qquad
 \alpha_-=\frac{\varphi}{2}-\frac{\theta}{4},
\]
and $\widetilde F: (0,\pi)\times\mathbb R^2 \rightarrow S^3(2)$ by
\[
 \widetilde F(r,\theta,\varphi)
 =
 \left(
 2\cos(r/2)e^{i\alpha_+},
 2\sin(r/2)e^{i\alpha_-}
 \right)\in S^3(2).
\]
A direct calculation shows that $\widetilde F^*g_{S^3(2)}=h_0$.

The angular periods determine the quotient precisely:
\[
 \widetilde F(r,\theta+2\pi,\varphi)
 =g\,\widetilde F(r,\theta,\varphi),
 \qquad
 \widetilde F(r,\theta,\varphi+2\pi)
 =g^2\,\widetilde F(r,\theta,\varphi).
\]
Thus $\widetilde F$ descends to a local isometry
\[
 F:P^\circ\longrightarrow M,
\]
with image
\[
 U=M\setminus(C_+\cup C_-),
\]
where
\[
 C_+=\{z_2=0\}/\Gamma,
 \qquad
 C_-=\{z_1=0\}/\Gamma.
\]
Moreover, $F$ is a diffeomorphism onto $U$. Indeed, away from
the coordinate circles, $r$ is determined by
\[
 |z_1|=2\cos(r/2),\qquad |z_2|=2\sin(r/2),
\]
and the inverse angular formulas are
\[
 \varphi=\alpha_++\alpha_-,\qquad \theta=2(\alpha_+-\alpha_-)\pmod{2\pi}.
\]
Changing either phase by $2\pi$, or replacing $(z_1,z_2)$ by
$g(z_1,z_2)$, leaves these inverse coordinates unchanged
modulo $2\pi$.

The group $\Gamma$ has order four and acts freely on $S^3(2)$.
Consequently $M$, with its quotient Riemannian metric, is a smooth
closed three-manifold of constant sectional curvature $1/4$.
Moreover, the intrinsic distance $d_U$ on $U$ agrees with the
restriction of the Riemannian distance $d_M$. To see this, join two
points of $U$ by a smooth path in $M$ whose length is arbitrarily close
to their $d_M$-distance. Since $C_+$ and $C_-$ are smooth embedded
circles of codimension two, this path can be perturbed, with endpoints
fixed and arbitrarily small change in length, to be transverse to both
circles. Such a transverse path misses both circles, since
$1+1<3$. Hence $d_U\le d_M$ on $U$; the reverse inequality is
immediate.
Since $U$ is dense in the complete space $M$, the map $F$
extends uniquely to an isometry from the metric completion
of $(P^\circ,d^\circ)$ onto $M$.

At the northern end, the surviving phase is $\alpha_+$ modulo
$\pi/2$, because the action of $g$ sends $\alpha_+$ to $\alpha_++\pi/2$.
Thus the coordinate on $C_+$ is
\[
 4\alpha_+=2\varphi+\theta=\Psi_+\pmod{2\pi}.
\]
Similarly, the coordinate on $C_-$ is
\[
 4\alpha_-=2\varphi-\theta=\Psi_-\pmod{2\pi}.
\]
The coordinate circles in $S^3(2)$ are great circles of
length $4\pi$. Their images in the quotient are embedded
closed geodesics of length $\pi$. Equivalently, their induced
metrics are
\[
 4\,d\alpha_+^2=\frac14\,d\Psi_+^2,
 \qquad
 4\,d\alpha_-^2=\frac14\,d\Psi_-^2.
\]
This identifies the two added circles and their metrics.
\end{proof}
\begin{remark}
For $x\in P^\circ$ with radial coordinate $r_x$, we have
\[
 d_M(x,C_+)=r_x,\qquad
 d_M(x,C_-)=\pi-r_x,\qquad
 d_M(C_+,C_-)=\pi.
\]
Contracting $C_+$ and $C_-$ separately therefore gives,
for regular points $x,y$, the quotient length distance
\[
 \min\left\{
 d_M(x,y),\ r_x+r_y,\ 2\pi-r_x-r_y
 \right\}.
\]
Since $d_M=d^\circ$ on the regular part, this is precisely
the defining distance of $Y_2$.

A small tubular neighborhood of either circle is a solid
torus. Contracting its core circle produces a neighborhood
homeomorphic to a cone on $\mathbb T^2$, so the corresponding
point of $Y_2$ is not a topological three-manifold point.
In contrast, $M$ is smooth everywhere. Thus $Y_2$ is not
isometric to the intrinsic completion.
\end{remark}

\subsection{Dimensions at least three}

For the angular Euclidean bundle
\(\underline\R\oplus TS^{n-1}\), the limiting connection at a fixed
level is
\begin{equation}\label{eq:angular-connection}
 D^\alpha_X(s,V)
 =\left(
 Xs-\alpha\langle X,V\rangle,
 \nabla_X^{S^{n-1}}V+\alpha sX
 \right),
 \qquad \alpha=\frac12\cos r.
\end{equation}
At the north end, \(\alpha=1/2\).

\begin{proposition}\label{prop:higher-completion}
For \(n\ge3\), the intrinsic metric completion of
\((P^\circ,h_0)\) is \(Y_n\). More precisely, every level
\(\{r\}\times L_n\) has intrinsic diameter tending to zero as
\(r\downarrow0\), and similarly as \(r\uparrow\pi\).
\end{proposition}

\begin{proof}
For a constant \(a\in(0,1)\), the curvature of \(D^a\) is
\[
 R^{D^a}(X,Y)(s,V)
 =\left(0,(1-a^2)
 (\langle V,Y\rangle X-\langle V,X\rangle Y)\right).
\]
Let $E=\underline{\mathbb R}\oplus TS^{n-1}$, and let
$e_0(u)=(1,0)$ denote the unit section of its trivial summand of .
Fix a base point $u_*\in S^{n-1}$ and choose an orthonormal basis
$e_1,\ldots,e_{n-1}$ of $T_{u_*}S^{n-1}$, identifying each $e_i$
with $(0,e_i)\in E_{u_*}$. Thus
$e_0(u_*),e_1,\ldots,e_{n-1}$ is an orthonormal basis of $E_{u_*}$.
For vectors $v,w$ in a fiber of $E$, write
\[
 (v\wedge w)\xi=\langle w,\xi\rangle v-\langle v,\xi\rangle w.
\]
Since $1-a^2\ne0$, the curvature endomorphisms at $u_*$
span all tangent rotations $e_i\wedge e_j$,
$1\le i<j\le n-1$. We
apply the Ambrose--Singer theorem (see \cite{KN}) to the restricted holonomy: its Lie
algebra is generated by curvature endomorphisms parallel transported to
the base point along piecewise smooth paths. Along a unit-speed
great-circle arc initially tangent to \(e_1\), let \(e_1(t)\) be its
tangent and extend \(e_j\), \(j\ge2\), by parallel transport in the
tangent bundle of \(S^{n-1}\). A \(D^a\)-parallel vector
\(s(t)e_0+v(t)e_1(t)\) satisfies
\[
 s'=av,\qquad v'=-as.
\]
If \(P_t\) denotes forward parallel transport, then
\[
 P_t^{-1}e_1(t)=-\sin(at)e_0+\cos(at)e_1,
 \qquad P_t^{-1}e_j(t)=e_j\quad(j\ge2).
\]
Transporting the tangent rotation \(e_1(t)\wedge e_2(t)\) back
therefore gives
\[
 (\cos(at)e_1-\sin(at)e_0)\wedge e_2.
\]
For \(\sin(at)\ne0\), subtracting
\(\cos(at)e_1\wedge e_2\) and dividing by \(-\sin(at)\) yields
\(e_0\wedge e_2\). Brackets with tangent rotations give every
\(e_0\wedge e_j\), and hence the holonomy Lie algebra is
\(\mathfrak{so}(n)\). The connection is metric, so restricted holonomy
lies in \(\SO(n)\); since it is connected, it equals \(\SO(n)\). The
full holonomy is a subgroup of \(\SO(n)\) containing the restricted
holonomy, and is therefore also \(\SO(n)\).

Full holonomy gives horizontal transitivity on the angular oriented
frame bundle \(L_n\). Let \(d_r\) be the intrinsic distance on the
level \(\{r\}\times L_n\), and fix \(p,q\in L_n\). Choose a piecewise
smooth \(D^{1/2}\)-horizontal curve from \(p\) to \(q\), with projected
curve \(\gamma\) in \(S^{n-1}\). Starting at \(p\), lift \(\gamma\)
horizontally using \(D^{(1/2)\cos r}\), and call the endpoint \(q_r\).
Continuous dependence of parallel transport gives \(q_r\to q\). The
lifted curve has length
\[
 f_0(r)\operatorname{length}_{S^{n-1}}(\gamma)\longrightarrow0.
\]
The points \(q_r,q\) lie in the same compact frame fiber and their
vertical distance tends to zero. Thus \(d_r(p,q)\to0\) for every fixed
pair.

We can upgrade this to uniform diameter collapse. Choose a fixed
Riemannian metric \(k\) on compact \(L_n\). For all sufficiently small
\(r\), the level quadratic forms are bounded above by \(C_nk\), and
hence
\[
 d_r(a,b)\le\sqrt{C_n}\,d_k(a,b).
\]
Given \(\eta>0\), choose a finite \(\eta\)-net
\(\{p_1,\dots,p_N\}\) in \((L_n,k)\). Pointwise convergence for the
finitely many pairs gives, for small \(r\),
\[
 d_r(p_i,p_j)<\eta
 \qquad(1\le i,j\le N).
\]
For arbitrary \(p,q\), choose \(p_i,p_j\) within \(k\)-distance
\(\eta\). Then
\[
 d_r(p,q)<2\sqrt{C_n}\eta+\eta.
\]
Letting \(\eta\downarrow0\) proves
\(\diam(L_n,d_r)\to0\). At the south end the limiting connection
is \(D^{-1/2}\). The orthogonal bundle map
\(C(s,V)=(-s,V)\) satisfies
\(D^{-a}_X(C(s,V))=C(D^a_X(s,V))\), so it conjugates the two
connections and their holonomy groups. Thus \(D^{-1/2}\) also has
full \(\SO(n)\) holonomy, and the same argument proves diameter
collapse as \(r\uparrow\pi\).

For \(x,y\in P^\circ\), descend radially to a northern level \(r=t\),
cross that level, and return radially. Then
\[
 d^\circ(x,y)
 \le(r_x-t)+(r_y-t)+\diam(L_n,d_t).
\]
Letting \(t\downarrow0\) gives \(d^\circ(x,y)\le r_x+r_y\). The
southern estimate similarly gives
\(d^\circ(x,y)\le2\pi-r_x-r_y\). Therefore the shortcut minimum
\eqref{eq:limit-distance} agrees with the intrinsic completion metric,
and the completion is \(Y_n\).
\end{proof}

\begin{remark}
The dimension-two discrepancy records holonomy created inside the
shrinking smoothing caps and invisible in the punctured limiting
metric. In dimensions \(n\ge3\), the limiting angular connection
already has full holonomy, so its intrinsic level diameters collapse.
\end{remark}

\section{On alternatives to Yamaguchi's original conjecture}


It remains to find meaningful alternatives to Yamaguchi's original conjecture. In view of the results of the present paper, it is, e.g., definitely not enough to replace topological manifolds by orbifolds. Instead, any substitute must inevitably take into account and control the holonomy phenomena caused by the possible presence of nontrivial short loops.

As a first step in this direction, let us present the following result. As before,
we use the canonical connection metrics, where the horizontal metric is the base metric, and the vertical metric is a fixed bi-invariant metric $b$ on $G = \mathrm{SO}(n)$.

\begin{theorem}[Uniformly small holonomy of short loops]\label{thm:small holonomy}
Let $(M_i, g_i)$ be closed, connected, oriented $n$-manifolds, $n \ge 2$, satisfying
\[
\sec_{g_i} \ge \kappa, \quad \diam(M_i, g_i) \le D, \quad \Vol(M_i, g_i) \ge v > 0.
\]
For piecewise smooth closed curves, define
\begin{equation}
\Omega(r) := \sup_i \sup_{\substack{x \in M_i,\; \gamma(0)=\gamma(1)=x \\ L_{g_i}(\gamma) \le r}} d_G(\Hol_\gamma, e).
\end{equation}
Here holonomy is represented in an oriented frame over $x$, and since $d_G$, the metric induced by the Riemannian metric $b$, is bi-invariant, its distance from the identity $e$ is independent of the chosen frame.

Then, under the assumption that $\Omega(r) \to 0$ as $r \downarrow 0$,
every Gromov--Hausdorff limit $Y$ of the oriented frame bundles $P_i = F^+(M_i, g_i)$ is a topological manifold of dimension $m = n(n+1)/2$. Moreover, after passing to an equivariantly convergent subsequence,
\[
Y \longrightarrow X = \lim_i (M_i, g_i)
\]
is a locally trivial topological principal $\mathrm{SO}(n)$-bundle.
\end{theorem}

\begin{proof}
Uniformly finite $r$-nets of the bases and of the fixed compact group $G$ yield $2r$-nets of $P_i$ by horizontal lifting. Moreover, $\diam P_i \le D + \diam G$. The bases are precompact by the usual compactness theorems, so the $P_i$ are precompact as well. For $p\in P_i$ and $A, B\in G$, the right actions satisfy
\[
d_{P_i}(pA, pB) \le d_G(A, B).
\]
After taking a subsequence, these actions converge to a continuous isometric $G$-action on $Y$, and metrically we have that $Y/G = X$.

Now suppose that $yh = y$ and $p_i \to y$. Then $d_{P_i}(p_i, p_i h) \to 0$. Choose connecting curves $c_i$ of lengths $\ell_i \to 0$. Their projections $\gamma_i$ are closed base curves with $L(\gamma_i) \le \ell_i$. Decomposing $c_i$ into a horizontal lift and a vertical group motion gives
\[
d_G(h, \Hol_{\gamma_i}) \le \ell_i, \quad d_G(h, e) \le \ell_i + \Omega(\ell_i) \longrightarrow 0.
\]
Thus $h = e$, and the limit action is free.

Finally, Gleason's theorem \cite{Gle} for free actions of compact Lie groups on metrizable spaces makes $\pi: Y \to X$ a locally trivial principal bundle. Noncollapse and Perelman's stability theorem \cite{Kap} imply that $X$ is a topological $n$-manifold. In a sufficiently small coordinate neighborhood \(U\subset X\), we have \(\pi^{-1}(U)\cong U\times G\). Thus, $Y$ is a topological $m$-manifold.
\end{proof}
\begin{remark}
The condition $\Omega(r)\to0$ as $r\downarrow0$ is equivalent to
the existence of a uniform constant $C<\infty$ such that
\begin{equation}\label{eq:linear estimate}
d_G(e,\Hol_{\gamma})\le CL_{g_i}(\gamma)
\end{equation}
for all $i$ and all based piecewise smooth loops $\gamma$.

Indeed, the linear estimate gives $\Omega(r)\le Cr$. Conversely,
choose $\delta>0$ with $2\delta<\operatorname{inj}(G,b)$ and
$r_0>0$ such that $\Omega(r_0)<\delta$.
For a loop of length $\ell\le r_0$, write
$h=\Hol_{\gamma}=\exp A$, where
$\theta=|A|=d_G(e,h)<\delta$.
If $\theta>0$, set $k=\lfloor\delta/\theta\rfloor+1$. Then
\[
\delta<d_G(e,h^k)=k\theta<2\delta.
\]
Since the $k$-fold traversal has length $k\ell$ and holonomy $h^k$,
we must have $k\ell>r_0$, whence
$\theta<2\delta\ell/r_0$.
For $\ell>r_0$, use $d_G(e,h)\le\operatorname{diam}(G,b)$.
Thus the estimate holds with
$C=r_0^{-1}\max\{2\delta,\operatorname{diam}(G,b)\}$.
\end{remark}

\begin{remark}
Under the standing assumptions
$\sec_{g_i}\ge\kappa$, $\diam(M_i,g_i)\le D$, and
$\Vol(M_i,g_i)\ge v>0$, each of the following conditions implies
$\Omega(r)\to0$ and also the linear estimate \eqref{eq:linear estimate}:
\begin{enumerate}
\item a uniform injectivity-radius bound
      $\operatorname{inj}(M_i,g_i)\ge i_0>0$;
\item a uniform conjugate-radius bound
      $\operatorname{conj}(M_i,g_i)\ge c_0>0$;
\item for some fixed $q>n$ and $\Lambda<\infty$,
\[
\int_{M_i}(\operatorname{Scal}_{g_i}^{+})^q\,
d\operatorname{Vol}_{g_i}\le\Lambda,
\qquad
\operatorname{Scal}^{+}=\max\{\operatorname{Scal},0\}.
\]
\end{enumerate}

Indeed, in (1), the sectional-curvature and injectivity-radius
bounds imply two-sided Hessian estimates for suitable distance
functions \cite[Theorem~3.5, p.~180]{PetersenSurvey}, yielding local orthonormal frames on balls of uniform
radius with uniformly bounded connection forms. In (2),
Bishop--Gromov comparison supplies uniform local noncollapse. Then the injectivity-radius estimate under a lower Ricci bound and
a positive conjugate-radius bound reduces this case to (1) (see \cite[Example~5.11, p.~191]{PetersenSurvey}). In (3), the sectional-curvature lower bound implies
\[
|\operatorname{Rm}_{g_i}|
\le C(n,\kappa)(1+\operatorname{Scal}_{g_i}^{+}).
\]
Together with the uniform upper volume bound, this leads to a uniform
$L^q$ curvature bound.
Local noncollapse and harmonic-coordinate estimates, by \cite[Theorem~5.4, p. 187]{PetersenSurvey}, then give
uniformly elliptic charts of uniform size with uniform $W^{2,q}$
bounds on the metric coefficients. Since $q>n$, Sobolev embedding
on smaller charts results in uniform bounds on their first derivatives,
and hence on the connection forms in the associated orthonormal
frames.

In all three cases, the parallel-transport equation yields
\eqref{eq:linear estimate} for sufficiently short loops. Meanwhile, the boundedness of
$d_G$ handles longer loops. Consequently, $\Omega(r)\le Cr$. The constant $C$ depends only on $n$, $\kappa$, $D$, $v$,
the fixed metric $b$, and the constants in the respective
additional condition.
\end{remark}
\clearpage

\begin{center}
\footnotesize{ACKNOWLEDGMENTS}
\end{center}
The first author is very grateful to Professor Burkhard Wilking for his enlightening lectures on advanced differential geometry and for suggesting the study of collapsing theory with lower curvature bounds. He also sincerely thanks Professor Wilking for his warm hospitality and generous help during his stay in M\"{u}nster.

\section*{AI Assistance Disclosure}
AI was used during the preparation of this manuscript
to assist with English-language editing, as well as
the organization of relevant literature. The authors take full
responsibility for the correctness and content of the manuscript.

\end{document}